\documentclass[11pt]{amsart}
\usepackage{amscd}
\usepackage{amsfonts}
\usepackage{amsmath,amsthm,hyperref}
\usepackage[all]{xy}
\usepackage{amsmath,amsthm,hyperref}
\usepackage{amsmath,amssymb,amsthm,latexsym}
\usepackage{amscd}
\usepackage{xcolor}
\usepackage{bm}
\usepackage{setspace}
\usepackage{caption}
\newtheorem{theorem}{Theorem}[section]
\newtheorem{proposition}[theorem]{Proposition}
\newtheorem{definition}[theorem]{Definition}
\newtheorem{corollary}[theorem]{Corollary}
\newtheorem{lemma}[theorem]{Lemma}

\numberwithin{equation}{section}
\theoremstyle{remark}
\newtheorem{remark}[theorem]{Remark}

\newcommand{\R}{\mathbb{R}}
\newcommand{\C}{\mathbb{C}}
\newcommand{\D}{\mathbb{D}}

\newcommand{\Up}{\mathbb{H}}

\newcommand{\dd}{\mathrm{d}}

\newcommand{\magenta}{\textcolor{magenta}}

\begin{document}
\title{\bf{Classification of homogeneous Willmore surfaces in $S^n$}}
\author{Josef Dorfmeister,   Peng Wang }
\address{Fakult\" at f\" ur Mathematik, TU-M\" unchen, Boltzmann str. 3, D-85747, Garching, Germany.}
\email{dorfm@ma.tum.de}
\address{College of Mathematics and Informatics, FJKLMAA, Fujian Normal University, Fuzhou 350117, P. R. China.}

\address{School of Mathematical Sciences, Tongji University, Shanghai 200092, P. R. China.} \email{netwangpeng@hotmail.com, netwangpeng@tongji.edu.cn}

\maketitle

\begin{abstract} In this note we consider homogeneous Willmore surfaces in $S^{n+2}$. The main result is that a  homogeneous Willmore two-sphere is conformally equivalent to a homogeneous minimal two-sphere in $S^{n+2}$, i.e., either a round two-sphere or one of the Bor\r{u}vka-Veronese 2-spheres in $S^{2m}$. This entails a classification of all Willmore $\C P^1$ in $S^{2m}$. As a second main result we show that there exists no homogeneous Willmore upper-half plane in $S^{n+2}$ and we give, in terms of  special constant potentials, a simple loop group characterization of all homogeneous surfaces which have an abelian transitive group.
\end{abstract}



\section{Introduction}

Homogeneous Willmore surfaces are the simplest ones among all Willmore surfaces.
Ejiri \cite{Ejiri1982} constructed the first homogeneous Willmore torus which is not conformally equivalent to any minimal surface in space forms. In \cite{J-L}, homogeneous  minimal tori in $S^{2n+1}$ and in $\mathbb C P^n$ are discussed systematically.
So far
there has not been any  systematic discussion of homogeneous Willmore surfaces. Very recently homogeneous Willmore 2-spheres have been studied by Pedit, Ma and Wang in terms of a variational method \cite{PeMaWa}. In this note we  provide a description of all homogeneous Willmore surfaces in spheres in terms of the loop group theory developed in \cite{DW1}. As an application, we derive a classification of homogeneous Willmore 2-spheres in $S^{n+2}$. They turn out to be exactly the only homogeneous minimal 2-spheres. Using the loop group theory, it is also easy to derive a characterization of homogeneous Willmore complex planes (including tori) in terms of potentials. Note there exist examples of homogeneous Willmore tori and planes which are not minimal in any space
form, which is different from the 2-sphere. See e.g. \cite{DW1}, \cite{Ejiri1982} and \cite{Li-V}. Moreover, we also show that there exists no homogeneous Willmore upper half plane.

The paper is organized as follows: Section 2 is a collection of basic results on homogeneous Riemann surfaces. In Section 3 we derive a description of homogeneous Willmore planes with abelian group action. Then we recall some basic results on the loop group theory of Willmore surfaces in Section 4. Section 5 provides a non-existence theorem of homogeneous Willmore upper half plane. Then we end this paper with a proof of the classification of  homogeneous Willmore 2-spheres in spheres.

The present paper was basically finished in 2012. It was finalized after the second named author had written separately a note on a partial result of this paper.


\section{Basic Results}

\subsection{Introductory definitions and results}

The notion of a homogeneous Willmore surface can be given in several ways, We adopt here the most direct one. Note that for a Riemannian manifold we denote by $Conf(M)$ the group of conformal diffeomorphisms.

\begin{definition}Let $X = H/H_0$ be a  connected homogeneous Riemann surface with a connected Lie group $H$ and a closed Lie subgroup $H_0$  of bi-holomorphic maps of $X$ and let $y: X \rightarrow S^m$ be a Willmore immersion. Then $y$ is called ``homogeneous" relative to $H$, if there exists a (continuous) homomorphism $R:H \rightarrow Conf(S^m)$ such that   $y(h.p) = R(h) y (p)$ for all $p \in X$ and all $h \in H$.
\end{definition}

Since $X = H/H_0$ is a two-dimensional real manifold with a Riemannian metric $g$ and a transitive conformal action of the connected group $H$, we can assume that either
$X \cong \R P^2$ or $X = H/H_0$ is a Riemann surface and $H$ a transitive group of bi-holomorphic maps.

From the classification of homogeneous Riemann surfaces we obtain:

\begin{enumerate}
\item The universal cover $\tilde{X}$ of X is either the unit sphere $S^2$, the unit disk
$\mathbb{E}$ (which we will frequently replace by  the  bi-holomorphically equivalent  upper half-plane $\mathbb{H}$),
or the whole complex plane.

\item the three conformal types of cylinders,
$\C^* = \C \setminus \{0\}$,  $\mathbb{E} \setminus {0}$ and $\mathbb{A}_{a,b}$,  where $\mathbb{A}_{a,b}$  denotes the
annulus $0 < a < |z| < b.$

\item the tori  $\mathbb{T} \cong \C/\mathcal{L}$, where $\mathcal{L}$ is a rank two lattice in $\C$.
\end{enumerate}

\begin{remark}\
\begin{enumerate}
 \item The cylinders above can be realized naturally in the form
$\C / \mathbb{Z}$, $\mathbb{H}/ T$ and $\mathbb{H}/ M$ respectively, where  $\mathbb{Z}$ means translations by integers in $x-$direction, and $T$ and $M$ are the discrete subgroups of $SL(2,\R)$  generated by  the matrices
\begin{equation}
t_0 = \left(
\begin{array}{cc}
1 & b\\
0&1\\
\end{array}
\right)
\hspace{2mm} \mbox{and} \hspace{2mm}
m_0 = \left(
\begin{array}{cc}
c & 0\\
0&c^{-1}\\
\end{array}
\right)
\end{equation}
\item Tori  consist of many conformal equivalence classes.

\end{enumerate}
\end{remark}
  It turns out that for each of the spaces listed above any  (connected) transitive Lie group  of bi-holomorphic automorphisms contains a transitive subgroup of a special type. Below we list these groups. The result is well known.
\begin{theorem} \label{homRiem}
Let $X$ denote any of the homogeneous Riemann surfaces listed above.
\begin{enumerate}
\item If $X \cong S^2$, then each transitive Lie group of bi-holomorphic maps contains a conjugate of $SU(2)/{ \{\pm I \}}.$

\item If $X \cong \mathbb{H}$, then each transitive Lie  group of bi-holomorphic maps contains a conjugate of the group $\Delta$ of real upper triangular matrices of determinant $1$ and positive diagonal elements.

\item If $X \cong \C$, then each transitive Lie group of bi-holomorphic maps contains the group of all translations.

\item If $X = \C^*$, then each transitive Lie group  of bi-holomorphic automorphisms contains  $\C^*$ acting by multiplication.

\item  If $X \cong \mathbb{T} \cong \C/\mathcal{L}$, then each transitive  Lie  group of bi-holomorphic maps contains the group of all translations.
\end{enumerate}
\end{theorem}

The original question of dealing with conformal transformations has turned a question of dealing with bi-holomorphic automorphisms of Riemann surfaces.
In particular, if $X = H/H_0$ is a Riemann surface on which the  connected  group $H$ acts bi-holomorphically and  transitively and $y : X \rightarrow S^m$ is a homogeneous Willmore surface relative to $H$, then the natural projection $\tilde{\pi}: \tilde{M} \rightarrow M$ yields a Willmore immersion $\tilde{y} :  \tilde{M} \rightarrow S^m$, given by
$\tilde{y} = y \circ \tilde{\pi}$.
Moreover, there exists some  connected transitive group $\tilde{H}$ of bi-holomorphic automorphisms of
$\tilde{X}$ which acts transitively on $\tilde{X}$ and satisfies
 \[ \tilde{y} ( \tilde{\gamma}.z) = R(\tilde{\gamma}) \tilde{y} (z) \hbox{ for all $z \in \tilde{M}$ and all $\tilde{\gamma} \in \tilde{H}$.}\]
As a consequence we obtain:
\begin{theorem}
If $y:X = H/H_0 \rightarrow S^m$  is a homogeneous Willmore surface different from $S^2$
 and $\tilde{\pi}$ the natural projection from the universal cover $\tilde{X}$ to  $X$, then the  lift  $\tilde{y} = y \circ \tilde{\pi}$ is a homogeneous Willmore surface defined on $\tilde{M}$. In particular, $\tilde{y}$ satisfies
\begin{equation}
\tilde{y}(\tilde{\gamma}.z) = R(\gamma)\tilde{y}(z) \hbox{ for all $z \in \tilde{M}$ and all $\tilde{\gamma} \in \tilde{H}$,}
\end{equation}
and $R$ is a homomorphism of Lie groups.
\end{theorem}

\section{The case of universal cover $\C$}

In this section we deal with the  cases $(3)$ and   $(5)$   above. Let $H$ be a connected Lie subroup of the group of bi-holomorphic automorphisms of $\C$. Then  Theorem \ref{homRiem} implies that  in this case there always exists a connected abelian transitive Lie group.

\begin{theorem}\
\begin{enumerate}
\item
If $y : \C \rightarrow S^{n+2}$ is a homogeneous Willmore surface relative to a group $H$ of translations, then $y$ can be obtained from some constant Maurer-Cartan form $\alpha = \mathcal{A}(\lambda)\dd u + \mathcal{B}(\lambda)\dd v$ of the moving frame satisfying
\begin{equation}\label{eq-constantmc}
 [\mathcal{A}(\lambda),\mathcal{B}(\lambda)]=0.
\end{equation}
\item Conversely, if $\alpha = \mathcal{A}(\lambda)\dd u + \mathcal{B}(\lambda)\dd v$ has the same form as the Maurer-Cartan form in Proposition 2.2 of \cite{DW1} for each $\lambda\in S^1$, with $\mathcal{A}(\lambda)$ and $\mathcal{B}(\lambda)$ being constant in $(u,v)$ and satisfying \eqref{eq-constantmc},  then
\begin{equation}\label{eq-F-C}
F(z, \bar{z},\lambda ) = e^{u\mathcal{A}(\lambda) + v \mathcal{B}(\lambda)}
\end{equation}
is the extended frame of the conformal Gauss map of a homogeneous  Willmore immersion $y:\C \rightarrow S^{n+2}$. In particular, $y$ can be read off from $F(z, \bar{z},\lambda )$ by Proposition 2.2 of \cite{DW1}.
\end{enumerate}
\end{theorem}
\begin{proof} One can apply Proposition 2.2 of \cite{DW1} and obtain that the Maurer-Cartan form $\alpha$ of an extended frame $F(z,\lambda)$ of $y$ satisfying $F(0,\lambda)= I$ is of the form $\alpha = \mathcal{A}(\lambda)\dd u + \mathcal{B}(\lambda)\dd v$.  The integrability condition then yields that $\mathcal{A}(\lambda)$ and $\mathcal{B}(\lambda)$ commute. The  property of being equivariant relative to all translations  implies   $F(z+z_0,\lambda)|_{\lambda=1} = R(z_0) F(z,\lambda)|_{\lambda=1}$ for any $z_0\in\C$. As a consequence, one has
\[F(u+vi,\lambda)|_{\lambda=1} = \exp(uE_1 + vE_2)) F(0,\lambda)_{\lambda=1}=\exp(uE_1+vE_2) \]
for some $E_1$ and $E_2$. Hence $\mathcal A(\lambda)|_{\lambda=1}=E_1$ and $\mathcal B(\lambda)|_{\lambda=1}=E_2$. In particular, $\mathcal{A}(\lambda)$ and $\mathcal{B}(\lambda)$ are constant in $z$ and $\bar z$.

The converse part is also a straightforward application of Proposition 2.2 of \cite{DW1}. From the form of $F$ stated in \eqref{eq-F-C} above it follows that  $y$ is homogeneous.
\end{proof}

In the case just discussed one only needs to quote a result which involves the loop group technique for Willmore surfaces.
For the other cases we will need to involve the loop group technique in more detail.

\section{The loop group formalism}

From now on, let $y:X \rightarrow S^{n+2}$ denote a homogeneous Willmore surface.
Then with $\tilde{X}$ denoting the universal cover of $X$ we obtain a homogeneous Willmore immersion $\tilde{y}: \tilde{X}  \rightarrow S^{n+2}$.
We recall briefly from \cite{DW1} the basic facts about the loop group
approach to Willmore immersions.

If $\tilde{X}$ is non-compact, then there exists a smooth extended frame
$F(z, \bar{z},\lambda)$, the frame of the conformal Gauss    map $f$ of $\tilde{y}$.
Moreover, $f: \tilde{X} \rightarrow SO^+(1, n+3)/{SO^+(1,3) \times SO(n)}$
is an associated family of conformally harmonic maps, where $f=F\mod SO^+(1, n+3)/SO^+(1,3) \times SO(n)$ for some frame $F$ of $f.$
Every harmonic map from $\tilde{X}$ to some symmetric space can be derived from some holomorphic potential as well as from some normalized potential
(For notation see \cite{DW1}).

In the case of $X=S^2$ there exists some meromorphic potential, a meromorphic one form on $S^2,$ with the appropriate behaviour in $\lambda$
which induces the conformal Gauss map $f$ associated  with the Willmore immersion $y$.  Moreover, the monodromy  at each pole of the potential is trivial $(= \pm I)$.
For more details concerning this case see \cite{DW1,DW2}.

So far we have only discussed potentials for Willmore immersions from $X$
  (and its universal cover $\tilde{X}$)  to $S^{n+2}$. The assumed homogeneity imposes additional restrictions.
We have seen above that there exists some subgroup $H$ of the group of all conformal transformations of  $\tilde{X}$) which acts transitively on  $\tilde{X}$ and for which we have:
\begin{equation}
R y(z) = y(h(z)), \hspace{4mm} \mbox{ for all } \hspace{2mm} z \in  \tilde{X}, ~R\in H.
\end{equation}

This induces (see e.g. \cite{DoHa;Sym1,DoHa;Sym2,Do-Wa-sym} )  the relation
\begin{equation}
 \chi (R,\lambda) F(z, \bar{z}, \lambda)  = F(h(z), \overline{h (z)} , \lambda)
k(z, \bar{z}), \hspace{4mm} \mbox{ for all } \hspace{4mm} z \in  \tilde{X},
\end{equation}
where $R\in H$ is arbitrary and $h = h_R \in H$ is appropriately chosen and $k$ is a $\lambda-$independent matrix function with values in the stabilizer group $K = SO^+(1,3) \times SO(n)$.
 On the level of  potentials we consider on the universal cover  $\tilde{X}$ of $X$  a potential which is the pull-back of some one-form defined on $X$. Then the
fundamental group $\pi_1(X)$ acts on  $\tilde{X}$ and induces the following conditions
\begin{equation} \label{inv-h}
h^* \eta = \eta \hspace{3mm} \mbox{for all} \hspace{3mm} h \in \pi_1 (X).
\end{equation}

In the orientable case, this has some restrictive meaning only in the case of  cylinders and tori. Later a similar formula will apply to the only non-orientable
homogeneous case.

 {We will now continue to discuss the homogeneous Willmore surfaces according to what universal cover they have. We recall, that the case of $\tilde{X} = \C$ has already been settled above.}

\section{The case of $\tilde{X} = \mathbb{H}$}

In this case, we can assume w.l.g. that the group $H$ contains at least the group
 \begin{equation}H_1=\left\{\gamma=\left(
                           \begin{array}{cc}
a&b \\
0&a^{-1}
\end{array}
\right)
|a\in\R^+, b\in \R\right\},\hbox{
with
}
 \gamma.z = \frac{az + b}{0z + a^{-1}} = a^2 z + ab.
 \end{equation}
So we have
 \begin{equation}
 y(\gamma.z) = \chi(\gamma) y(z)
 \end{equation}
for $\gamma$. Set
\[\sigma_3 =  \left(
                           \begin{array}{cc}
1&0 \\
0&-1
\end{array}
\right)
\hspace{2mm}\mbox{and}\hspace{2mm}
\nu = \left(
                           \begin{array}{cc}0&1\\
0&0
\end{array}
\right).\]
We have
\[y( e^{t\sigma_3}.z) = e^{tB} y(z) \hspace{3mm} \mbox{and} \hspace{3mm}  y( e^{s \nu}.z) = e^{sD} y(z).\]
Here $B$ and $D$ are the images of $\sigma_3$ and $\nu$ by  the
 monodromy representation $\chi$ respectively.
 Note: we obtain
 \begin{equation} \label{comm}
 [B,D] = 2 D.
\end{equation}
It is straightforward to compute
\[(e^{ 1/2 \ln (v) \sigma_3} e^{v^{-1}u \nu}).i=z.\]
Altogether we thus obtain
\begin{equation}
y(z) = y(u +iv) = e^{ 1/2 \ln (v) B} e^{v^{-1}u D} y(i) =  e^{ u D} e^{1/2 \ln(v) B} y(i).
\end{equation}
Similarly we obtain for the frame (and equally well for the extended frame):
\begin{equation}
F(z) = F(u +iv) = e^{ 1/2 \ln (v) B} e^{v^{-1}u D} F(i) k(u,v) =  e^{ u D} e^{1/2 \ln(v) B} F(i) \hat{k}(u,v).
\end{equation}
 Now we apply what was mentioned at the end of the last section.
It is well known that one can classify, up to conjugation, which choices there exist for $B$ and $D$.
 In  formula (7) of \cite{PWZ},  it is shown that each non-abelian maximal solvable
  subalgebra of $so(1, p)$ is of the form $\mathfrak{s} = \R S \oplus \mathcal{N}$ with $S$ a semisimple matrix and $\mathcal{N}$ an abelian subalgebra  of nilpotent matrices  of dimension $p-1$ such that $[S,N] = N$ for all $N \in \mathcal{N}$.
Considering
\begin{equation} \label{solvable}
\mathcal{S} = \left\{\left(
                \begin{array}{ccc}
0& a &X^t \\
a& 0& -X^t \\
X &X & 0 \\
                             \end{array}
                           \right)\right\}.
\end{equation}
it is easy to verify the properties listed above for a maximal solvable subalgebra of $so(1,p)$.
Therefore $\mathcal{S}$ is maximal solvable.

 It follows that up to a conformal transformation of our original Willmore immersion we can assume that the monodromy group of our homogeneous Willmore upper half-plane has the form (\ref{solvable}).
In particular, $D$ has the form
\[\left(
                \begin{array}{ccc}
0&0 &X^t \\
0& 0& -X^t \\
X &X & 0 \\
                             \end{array}
                           \right).\]
On the other hand since $\exp(t\nu)$ is a translation along the $v-$axis, we can view $y$ as an equivariant surface along $v$ direction. So the corresponding monodromy gives a Delauney type matrix \cite{BK,DW-equ}. But a
                          Delauney type matrix can not be of the above form, since it is derived by the Maurer-Cartan form of $F(z,\lambda)$ \cite{DW-equ}.  Altogether we obtain
\begin{theorem}
There exists no homogeneous Willmore immersion from $\Up$ to $S^{n+2}$.
\end{theorem}

\section{The case of $X = S^2$}


In this section, we will show that a homogeneous Willmore 2-sphere is conformally equivalent to either the round 2-sphere or one of the Bor\r{u}vka-Veronese 2-spheres in $S^{2m}$. The main idea is to use the representations of $\mathfrak{so}(3)$ in $\mathfrak{so}{^+(1,n+3)}$ and loop group theory to show   that a representation is irreducible and hence the orbit gives one of the Bor\r{u}vka-Veronese 2-spheres.

We  introduce our main result in  Section 6.1. Then we consider the monodromy matrices of $S^2$ and their representations in Section 6.2. Using these matrices, in Section 6.3 we obtain the normalized potentials of the homogeneous Willmore 2-spheres as well as a property of the matrices. With these preparations, we can show in Section 6.4 that  the representation is necessarily irreducible and obtain the  proof of our main result. Finally we also consider the antipodal symmetries of the surfaces and show whether they will reduce to $\R P^2$ or not.

\subsection{The main result}

\begin{theorem}\label{thm-s2}
 Let $y :M \rightarrow S^{n+2}$ be a homogeneous and full Willmore $S^2$, $n\geq1$.  Then  $n=2m-2$ and $y$ is conformally equivalent to the full  Bor\r{u}vka-Veronese 2-sphere in $S^{2m}$.
\end{theorem}

The theorem is a corollary of the following result and the classical results of Bor\r{u}vka \cite{Boruvka} and Calabi \cite{Calabi} describing homogeneous minimal two-spheres in $S^{n+2}$.
 \begin{proposition}\label{prop-minimal}Each full homogeneous Willmore two-sphere is conformally equivalent to a minimal two-sphere in $S^{n+2}$.
\end{proposition}

In Section 6.4 we will provide a proof of Proposition \ref{prop-minimal}. Before proceeding, let's first explain the idea of the proof. Since the surface is homogeneous, it can be viewed as an orbit of a group $SO(3)\subset SO^+(1,n+3)$. This induces a representation of  $\mathfrak{so}(3)$ in $\mathfrak{so}{^+(1,n+3)}$. Using this representation, we can describe the Maurer-Cartan  form of the surface $y$.
The Willmore property allows us to introduce the loop parameter $\lambda$ into the M-C form. Moreover, we can obtain the normalized potential of $y$.
 Finally with the help of the loop parameter, we can show that the representation is irreducible and hence we can determine the form of the normalized potential of $y$. Then by Theorem 1.1 of \cite{Wang;minimal},  the surface is minimal in $S^{n+2}$.

\subsection{ The monodromy matrices}
To begin with, let us first recall some basic results concerning $SO(3)$. Let $T_j$ be a basis of $\mathfrak{so}(3)$ which satisfies the commutation relations (see \cite{Hall})
 \begin{equation} \label{comrel}
 [T_3, T_2] = - T_1, \hspace{3mm}  [T_3, T_1] = T_2, \hspace{3mm} [T_1 , T_2] =  T_3.
 \end{equation}
Actually we assume ( setting $S^2=\mathbb{C}\cup\{\infty\}$):
\begin{enumerate}
\item The matrix $T_1$ generates rotations about $1\in\C$: $e^{tT_1}.z = \frac{ z \cos (\frac{1}{2} t) - i\sin (\frac{1}{2} t)}{ - z i \sin (\frac{1}{2} t) +\cos (\frac{1}{2} t)}$;
\item The matrix $T_2$ generates rotations about $i\in\C$: $e^{tT_2}.z = \frac{ z \cos (\frac{1}{2} t) - \sin (\frac{1}{2} t)}{ z \sin (\frac{1}{2} t) +\cos (\frac{1}{2} t)}$;
\item The matrix $T_3$ generates rotations about $0\in\C$: $e^{tT_3}.z = e^{-it}z.$
\end{enumerate}

If we have a symmetry $y(g.z) = U y(z)$ of some Willmore surface into $S^{n+2}$
with $U \in SO(n+3)$, then the corresponding matrix in $SO^+(1, n+3),$ acting on the Grassmannian, the frames etc,  is naturally of the form
\[
\left(
                   \begin{array}{cc}

1 & 0 \\
0 & U \\
                   \end{array}
                   \right).
                  \]
                   For a homogeneous Willmore $S^2$, we can assume that  the group acting on the Grassmannian is contained in the natural $SO(n+3)$ as above. So the monodromy representation $\chi_U( g, \lambda = 1)$  takes values in  $SO(n+3)$. Considering the representation  {$\chi_U(g,\lambda)$ of} $SO(3)$ in
 $\Lambda SO^+(1,n+3)_\sigma$  one can decompose it into irreducible ones.

 It is well known that all irreducible representations of $SO(3)$ act on odd-dimensional spaces and are uniquely determined by the eigenvalues of the image  $\chi_U(T_3,\lambda)$
 of  $ T_3 = \frac{1}{2} diag(-i,i) \in \mathfrak{su}(2)$  under $\chi_U$   (See e.g, Theorem 4.12 of \cite{Hall}).
 These eigenvalues all are of the form $ik,$ where  $k \in \mathbb{Z}$.
 Moreover, the irreducible summands can be read off from the multiplicities of the eigenvalues of $\chi_U(T_3,\lambda)$. As a consequence, we have

 \begin{corollary}\label{cor-irr}The monodromy representation of $SO(3)$ is irreducible on $\R^{n+3}$  if and only if the $0-$eigenvalue of  $\chi_U(T_3,\lambda)$, acting on $\R^{n+4}$, has multiplicity $2$.
\end{corollary}

Let $y:S^2\rightarrow S^{n+2}$ be a homogeneous Willmore two-sphere.
 Then the approach of \cite{DW1} requires to consider an extended frame, from which $y$ can be reproduced again. If $F_U(z, \lambda)$ denotes the extended frame associated with the representation $\chi_U$ above, then it turns out to be useful to also consider the extended frame
$\hat{F}(z, \lambda)  = F_U(0,\lambda)^{-1}F_U (z,\lambda).$ This yields the representation
\[\hat{\chi}(g,\lambda) = F_U(0,\lambda)^{-1}\chi_U(g,\lambda) F_U(0,\lambda)\] of $SO(3)$ in
$\Lambda SO^+(1,n+3)_\sigma$.  Note that the image of $SO(3)$ of $\chi$  is again compact, but not necessarily contained in $\Lambda SO(n+3)_\sigma$. But all statements about eigenvalues and irreducible representations made above still hold in the present situation as well.

Under the monodromy representation $\hat{\chi}( \cdot,\lambda)$ relative to $\hat{F}$ we define  (using, by abuse of notation, $\hat{\chi}$ for $\dd\hat{\chi}$):
\begin{equation}\label{eq-Aj}
\hat{\chi}(T_1, \lambda)=A_1(\lambda),~ \hat{\chi}(T_2,\lambda) =A_2(\lambda)~ \hbox{ an }~  \hat{\chi}(T_3, \lambda)=A_3.
\end{equation}
Note that by \cite{DW-equ} (also see below at the end of the roof of the next proposition), $A_3$ does not depend on $\lambda$ and is contained in
$\mathfrak{k}$ and in $\mathfrak{so}(n+3).$
More precisely, using also the notation
\begin{equation}
S  = \left(
\begin{array}{cc}
0 & 1 \\
-1 &0\\
\end{array}
\right).
\end{equation}
we will show

\begin{proposition} \label{lambdadistr}
The matrices $A_j(\lambda)$ only contain the powers $\lambda^{-1}, \lambda^{0},$ and
$\lambda^{1}.$ Moreover, the matrices $A_1(\lambda)$ and $A_2(\lambda)$ have the form of a generator of a translationally equivariant Willmore surface and $A_3$ has the diagonal block form $A_3 = (0, S, A_{3,2}),$ with the $2 \times 2$ matrices $0$ and $S$ and the $n \times n$ matrix $A_{32}$.
\end{proposition}
\begin{proof}
Let's consider the matrix $A_1(\lambda)$. We remove $1$ from $\C$ and consider the universal cover $\tilde{M} \cong \C$ of $M = \C \setminus \{1\}$. We can do this such that the covering map maps $0$ to $0$ and that the action of the one parameter group $e^{tT_1}$ lifted to $\tilde{M}$ acts by translations parallel to the real axis. Lifting the extended frame to the universal cover we observe that it attains the value $I$ at $0,$ whence the monodromy is of the form $e^{t \tilde{A}_1}$ and by \cite{BK} it follows that $\tilde{A}_1(\lambda)$ has the form as claimed. It is easy to see that actually $\tilde{A}_1(\lambda) = A_1(\lambda)$ holds.
The claim for $A_2(\lambda)$ follows verbatim.
Let's consider finally $A_3$. The relation
\begin{equation}
\hat{F}( e^{it}z ,\lambda) = \hat{F}( e^{-itT_3}z  ,\lambda) =
e^{-tA_3(\lambda)} \hat{F}( z ,\lambda) k_3(t,z)
\end{equation}
and the property $\hat{F}( 0 ,\lambda) =I$ imply
\begin{equation}
I = e^{-tA_3(\lambda)} k_3(t,0).
\end{equation}
From this we infer that $A_3(\lambda)$ actually is independent of $\lambda$ and contained in
$\mathfrak{k}$. From   $(4.7)$  of \cite{Do-Wa-sym} and the last equation we know how the first $4 \times 4$ block of $A_3$ looks like. Now it is straightforward to verify that $A_3$ has the form stated.
\end{proof}

\subsection{An Application of Wu's Formula}

Using the transformation properties of the one-parameter groups generated by the $T_j$ we introduce a new frame $\tilde F$ which permits to apply Wu's formula easily.
We have
$y(e^{i \theta} r) = y(e^{- \theta T_3}.z) = e^{-\theta A_3}. y(r).$
Moreover, considering the rotation about $i\in \C \subset S^2$, we have
$y(r)=e^{ -2\arctan(r) A_2} .y(0)$
and altogether we have the formula
\begin{equation} \label{trafoy}
y(e^{i \theta}r) = e^{-\theta A_3} e^{-2\arctan(r) A_2} .y(0).
\end{equation}
For the frame $\hat{F}$, which satisfies  $\hat F(z=0)=I$, we obtain
\begin{equation} \label{trafoF}
\hat{F}(e^{i \theta}r) = e^{-\theta A_3} e^{-2\arctan(r) A_2} k_2(0,r) k_3(\theta,r),
\end{equation}
where $k_2$ and $k_3$ are the factors occurring in formula  $(4.7)$  of \cite{Do-Wa-sym}.
These equations also hold after introducing $\lambda$ (note that there is no $\lambda$ inside the $k$'s !).
Since $  k_2(0,r) k_3(\theta,r ) \in SO^+(1,3)\times SO(n)$, we consider a new extended frame of $y$:
\begin{equation} \label{trafoF2}
\tilde F(e^{i \theta}r,\lambda) = e^{-\theta A_3} e^{-2\arctan(r) A_2(\lambda)} e^{\theta A_3}.
\end{equation}
Note that $\tilde F$ and $\hat{F}$ have the same monodromy representation, since they only differ by a gauge in $K$. Moreover, $\tilde F$ is a priori only defined, where the polar representation for $z$ is valid.
However, since $A_3$ is independent of $\lambda$ and has purely imaginary eigenvalues with integer imaginary parts, the frame $\tilde F$  actually is defined on $\C^*$.
The following proposition shows it is
 in fact  defined on $\C$.
\begin{proposition}
The frame $\tilde F$ has the form
\begin{equation} \label{Fsmooth}
\tilde  F(z,\lambda) = \tilde F(e^{i \theta}r,\lambda) =
 e^{-\theta A_3} e^{-2\arctan(r) A_2(\lambda)} e^{\theta A_3}
  = e^{  \frac{-2 \arctan(r)}{r} ( u A_2 + v A_1 )}.
  \end{equation}
In particular, $\tilde F$ is real analytic at $z=0$ and we have $\tilde F(0,\lambda) = I.$
As a consequence, $\tilde F$ and $\hat{F}$ have the same normalized potential and generate the same Willmore surface.
\end{proposition}

\begin{proof}
In view of (\ref{trafoF2}) it suffices to consider
\[E = e^{- \theta A_3} \arctan(r) A_2 (\lambda ) e^{ \theta A_3} =
  \arctan (r) exp( -\theta ad (A_3)  (A_2)) .\]
But by (\ref{comrel})
 this  unravels  to\[E =  \arctan(r) ( A_2\cos  \theta  + A_1 \sin  \theta) .\]
 Now in the upper half-plane  the angle $\theta$ of an element of $\C$ in the range $(-\pi, \pi)$ is given by $ \arccos (u/r)$ and by $ - \arccos (u/r)$  in the lower half-plane.
 Hence $E =  \arctan(r) ( \frac{u}{r} A_2 + \frac{\sqrt{r^2 - u^2}}{r} A_1)  =
 \frac{\arctan(r)}{r} ( u A_2 + v A_1)$ for $u + iv$ in the upper half-plane and
  $E =  \arctan(r) ( \frac{u}{r} A_2 - \frac{\sqrt{r^2 - u^2}}{r} A_1)  =
 \frac{\arctan(r)}{r} ( u A_2 + v A_1)$ for $u + iv$ in the lower half-plane.
 As a consequence,  $E$ is well defined on $\C$ and real-analytic there.
 Finally, it is clear now that $E(z=0) = 0,$ whence $\tilde F(0,\lambda) = I.$
  \end{proof}

As a consequence of (\ref{Fsmooth})   we can compute the normalized potential of the given Willmore immersion by using Wu's Formula \cite{Wu;normalized}.
In the statement of loc.cit one considers the Maurer-Cartan form of some frame for the given harmonic map (in our case the conformal Gauss map of the given Willmore surface).
Then, at least in some neighbourhood of the base point, say, $z = 0$, one can set
$ \bar z = 0,$ since the frame  is real-analytic.
From the resulting differential one-form one computes the normalized potential.

Actually, since the frame $\tilde F$ is real-analytic near the origin, one can set $\bar z = 0$ already directly in the frame. Doing this we obtain (by using  (\ref{Fsmooth}) )
\begin{equation*}
\tilde F(z, \bar z = 0, \lambda) =   e^{- zA_2 + iz A_1},
\end{equation*}
 since the real analytic function $\arctan(r) /r $ attains the value 1 at $r=0$.
Now,  the Maurer-Cartan form of $\tilde F(z, \bar z = 0, \lambda) $ is
\begin{equation} \label{holomorphicalpha}
hol(\alpha) = (-A_2 +  iA_1)\dd z.
\end{equation}
In view of Proposition \ref{lambdadistr} we can write
\begin{equation}\label{eq-A1-A2}
A_1=\lambda^{-1}H_1+H_0+\lambda \bar{H}_1 \hspace{2mm} \mbox{and} \hspace{2mm}
A_2=\lambda^{-1}L_1+L_0+\lambda \bar{L}_1.
\end{equation}
From this we derive
\begin{theorem}
We have
\begin{equation}\label{eq-wu-1}
L_1 = -i H_1,
\end{equation}
and the normalized potential $\xi$ of the given homogeneous Willmore surface is
\begin{equation} \label{normpot}
\xi (z) \dd z = e^{z \beta_0} \beta_1e^{-z \beta_0}  = e^{z ad(\beta_0)} \beta_1,
\end{equation}
where $\beta_0 = -L_0 +iH_0$ and $\beta_1 = \lambda^{-1}(-L_1 + i H_1)$.
\end{theorem}

\begin{proof}
The first claim follows from the fact that $hol(\alpha)$ does not contain the power $\lambda =
\lambda^1$, as stated in the proof of ``Wu's Formula", Theorem 2.1 of \cite{Wu;normalized}.
From the proof of loc.cit it is also clear that we need to consider
$\beta_1 = \lambda^{-1}( -L_1 + i H_1)$ and  $\beta_0 =  -L_0 + iH_0 $ and solve the ode $d \psi_0 = \psi_0 \beta_0 $ with $\psi_0(0) = I$.
Since this solution is $\psi_o(z) = e^{z\beta_0}$ the claim follows directly from loc.cit.
\end{proof}


\subsection{Uniqueness of the monodromy representation}

First we describe $A_3$ explicitly in view of Proposition \ref{lambdadistr}.
\begin{lemma} We have (up to a conjugation in $K$)
\begin{equation}\label{eq-A3-can}A_3=\hbox{diag}\{0_{2\times 2},S,S_j\}, \hbox{ with }
 \mathbb{S} = (S_0, S_1, \dots)  \hspace{2mm} \mbox{and} \hspace{2mm}
S_j=j \left(
                                                   \begin{array}{ccc}
                                                     S & 0 & 0 \\
                                                     0 & \cdots & 0 \\
                                                     0 & 0 & S \\
                                                   \end{array}
                                                 \right)_{2n_j\times 2n_j}
,~0\leq j\leq m.\end{equation}
\end{lemma}

\subsubsection{Irreducibility of the monodromy representation}
By Corollary \ref{cor-irr}, the monodromy representation is irreducible if and only if the $S_0-$term does not show up  in \eqref{eq-A3-can}.
Substituting \eqref{eq-A1-A2} and \eqref{eq-wu-1} into the commutation properties
$[A_3 ,A_2(\lambda)]=- A_1(\lambda),\ [A_3 ,A_1(\lambda)]= A_2(\lambda),\ [A_1(\lambda),A_2(\lambda)]= A_3,$
we obtain
\begin{lemma}\ \label{lemma-comm-1}
\begin{enumerate}
\item $[A_3, L_1]=-iL_1$, $[A_3,[A_3, L_1]]=-L_1$;
\item $[A_3, -L_0+iH_0]=-i(-L_0+iH_0)$, $[A_3,[A_3, -L_0+iH_0]]=-(-L_0+iH_0)$;
\item $[H_0, L_0]+2i[L_1,\bar L_1]=A_3$;
\item $[\bar L_1, -L_0+iH_0]=0$.
\end{enumerate}
\end{lemma}
We decompose $L_1$ and $-L_0+iH_0$ according to \eqref{eq-A3-can}
\[L_1=\left(
        \begin{array}{cc}
          0 & B_1 \\
          -B_1^tI_{1,3} & 0 \\
        \end{array}
      \right)
\hbox{ with } B_1=\left(
        \begin{array}{ccccc}
          B_{10} & \cdots & B_{1m} \\
           B_{20} &  \cdots & B_{2m} \\
        \end{array}
      \right)\]
      and \[
      -L_0+iH_0=
      \left(
        \begin{array}{cccccc}
         \hat R_{11} & \hat R_{00}  &  0 &\cdots  &0 \\
          -\hat R_{00}^tI_{1,1} & \hat R_{22}&  0 &\cdots  &0 \\
         0& 0 &R_{00}&\cdots  &R_{0m} \\
         \cdots &  \cdots  & \cdots  & \cdots  &\cdots  \\
         0& 0 &-R_{0m}^t&\cdots  & R_{mm} \\
        \end{array}
      \right).\]
Set
\[Q_1=\left(
        \begin{array}{cc}
          1& -i \\
          -1 & i \\
        \end{array}
      \right),\ \
Q_2=\left(
        \begin{array}{cc}
          1& -i \\
          i & 1 \\
        \end{array}
      \right) \hbox{ and } e_0=\left(
                                            \begin{array}{c}
                                              1 \\
                                              -i \\
                                            \end{array}
                                          \right).
\]

By using (1) and (2) of Lemma \ref{lemma-comm-1} and elementary computations, we derive
\begin{lemma}\
\begin{enumerate}
\item $B_{1j}=0$ for all $j\neq 1$ and
$B_{11}= \left(
                   \begin{array}{ccc}
                     a_{11}Q_1 & \cdots & a_{1,n_1}Q_1 \\
                   \end{array}
                 \right)
         ;
$
\item $B_{2j}=0$ for all $j\neq 0,2$, and
\[B_{20}=\left(
                   \begin{array}{cccccc}
                     b_{11}e_0 & \hat b_{11}e_0& \cdots &b_{1,n_0}e_0&\hat b_{1,n_0}e_0 \\
                   \end{array}
                 \right),\ B_{22}=\left(
                   \begin{array}{ccc}
                     c_{11}Q_2 & \cdots & c_{1,n_2}Q_2 \\
                   \end{array}
                 \right). \]
\item $\hat R_{11}=\hat R_{22}=0$ and $\hat R_{00}=\left(
                                                     \begin{array}{cc}
                                                       a & -ia \\
                                                       c & -ic\\
                                                     \end{array}
                                                   \right)
$;
$R_{jl}=0$, for all $j-l\neq1$, $j\leq l$, and
$R_{j,j+1}=\left(
              \begin{array}{ccc}
                q_{kl}^{(j)}Q_2 \\
              \end{array}
            \right)_{1\leq k\leq n_j, 1\leq l\leq n_{j+1}}
.$
\end{enumerate}
\end{lemma}

\begin{lemma}If $B_{20}=0$, then   $A_1(\lambda)$, $A_2(\lambda)$ and $A_3$ define an irreducible representation.
\end{lemma}
\begin{proof} In this case, by (3) of Lemma \ref{lemma-comm-1}, we have
$R_{01}\bar R_{01}^t-\bar R_{01} R_{01}^t=0.$
As a consequence we have
$\sum_{l=1}^{n_0}|q_{kl}^{(0)}|^2=0, \hbox{ for all } k=1,\cdots, n_0.$
That is, $R_{01}=0$. So all  $A_1(\lambda)$, $A_2(\lambda)$ and $A_3$ are of the form
\[\left(
    \begin{array}{ccc}
      \cdots& 0 & \cdots \\
      0 & 0 & 0 \\
      \cdots & 0 & \cdots \\
    \end{array}
  \right).
\]
As a consequence, the representation is in a smaller space, contradicting  the fullness of the surface. Hence $A_3$ has no $S_0$ term, i.e., $A_3$ has exactly two $0-$eigenvalues and by Corollary \ref{cor-irr} the representation is irreducible.
\end{proof}

So we only need to show
\begin{lemma} $B_{20}=0$.
\end{lemma}

\begin{proof}Suppose $B_{20}\neq 0.$
Recall that  for  Willmore surfaces one has
$B_1^tI_{1,3}B_1=0$.  As a consequence
\[B_{22}^tB_{20}=\left( c_{1j}b_{1k} e_0
                 \right)_{1\leq j\leq n_2,1\leq k\leq n_0}
=0.\]
Since $B_{20}\neq 0$, one of the   $b_{1k}$ satisfies $b_{1k}\neq0$ and hence $c_{1j}=0$ for all $j$, whence, $B_{22}=0$.

Moreover, by (4) of Lemma \ref{cor-irr}, we have
\[  -\hat R_{00}^tI_{1,1}\bar{B}_{11}+\bar B_{20}R_{01}=0 \hbox{ and } \bar{B}_{11}R^t_{01}+\hat R_{00}\bar B_{20}=0.\]
The last equation reads
\begin{equation}\label{eq-ac}
\left(
    \begin{array}{cccccc}
    \sum_{j=1}^{n_1}2a_{1j}q_{1j}^{(0)}Q_1 &\cdots &\sum_{j=1}^{n_1}2a_{1j}q_{n_0,j}^{(0)}Q_1 \\
    \end{array}
  \right)+2\left(
            \begin{array}{cccccc}
              a\bar{b}_{11} & a\bar{\hat b}_{11} & \cdots &  a\bar{b}_{1,n_0} & a\bar{\hat b}_{1,n_0}  \\
              c\bar{b}_{11} & c\bar{\hat b}_{11} & \cdots &  c\bar{b}_{1,n_0} & c\bar{\hat b}_{1,n_0}  \\
            \end{array}
          \right)=0.
\end{equation}
Next we claim that $|a|<|c|$ holds. As a consequence, \eqref{eq-ac} can not holds, since one of the  $b_{1k}$ satisfies $b_{1k}\neq0$.

 Now let's prove the claim.
 Consider the left upper  $4\times 4$ diagonal block of the equation (3) of Lemma \ref{lemma-comm-1}: $[H_0, L_0]+2i[L_1,\bar L_1]=A_3$.
We obtain that
\[|c|^2-|a|^2=1+2\sum_j(|b_{1j}|^2+|\hat b_{1j}|^2+2|c_{1j}|^2),\]
from which we see that $|c|^2-|a|^2\geq1>0$.

The above contradiction indicates $B_{20}=0$, which finishes the proof.

\end{proof}

{\em Proof of Proposition \ref{prop-minimal}:} From the proof of the above lemma we see that $|c|>|a|$ and $a/c\in \R$. Consider the hyperbolic rotation
$T_t=\hbox{diag}\{T_{1t},I\}\ \hbox{ with }
T_{1t}=\left(
    \begin{array}{ccccc}
      \cosh t& \sinh t  \\
      \sinh t & \cosh t \\
    \end{array}
  \right)$ and $a\cosh t+c\sinh t=0$. The new representation  $T_1A_jT_1^{-1}$ has the same form as $A_j$ except that now the new $\tilde a=0$. So w.l.g. we can assume that $a=0$. Now we see that all of $A_1(\lambda)$, $A_2(\lambda)$ and $A_3$ take values in $\mathfrak{so}(n+3)$ and as a consequence, the normalized potential \eqref{normpot} takes values in $\mathfrak{so}(n+3)$. By Theorem 1.1 of \cite{Wang;minimal}, $y$  is a
  minimal surface in $S^{n+2}$.\hfill$\Box$

\subsection{On homongeneous Willmore $\R P^2$}
Finally we consider the  case where a homogeneous  Willmore immersion from the two-sphere descends to a map defined on $\R P^2$, i.e., where the Willmore two-sphere is invariant under the fixed point free antiholomorphic involution  $\mu(z)=-\frac{1}{\bar z}$.
By \eqref{trafoy}, we obtain
\begin{equation} \label{trafoy}
y(\mu(z))=y(-e^{i \theta}\frac{1}{r}) = e^{-\theta A_3} e^{2\arctan(\frac{1}{r}) A_2} .y(0)=
 e^{-\theta A_3} e^{-2\arctan(r) A_2} e^{\pi A_2}.y(0).
\end{equation}
Here we have used the fact that $\arctan(\frac{1}{r})+\arctan(r)=\frac{\pi}{2}$.

Let's consider next the action $e^{\pi A_2}.y(0)$. It is a rotation of $y(0)$ by degree $m\pi$. Here $m$ is the degree of the representation and hence half of the dimension of $S^{2m}$ if $y$ is full in $S^{2m}$. Obviously, $e^{\pi A_2}.y(0)=y(0)$ if $m=2\tilde{m}$ and $e^{\pi A_2}.y(0)= - y(0)$ if  $m=2\tilde{m}+1$.

If $e^{\pi A_2}.y(0)=y(0)$, then we have
$y(\mu(z))=y(z).$
If $e^{\pi A_2}.y(0)=-y(0)$, we have
$y(\mu(z))=-y(z).$

Summing up we obtain:

\begin{corollary}Let $y$ be a homogeneous Willmore two-sphere.
\begin{enumerate}
\item If $y$ is full in $S^{4\tilde m}$, it descends to a Willmore immersion from $\R P^2$.
\item The full homogeneous Willmore two-sphere in $S^{4\tilde m-2}$  has an antipodal symmetry \[y(-\frac{1}{\bar z})=-y(z).\]
\end{enumerate}
\end{corollary}
\begin{remark} This can also be read off from the spherical harmonics functions. If the degree of these function is even, then the surface reduces to an immersion from $\R P^2$. Otherwise the surface has an antipodal symmetry as above. We believe this result is well-known to experts. But so far we do not have a reference in literature.
\end{remark}

\def\refname{References}

\end{document}